\documentclass{amsart}
\usepackage{graphics}
{\theoremstyle{plain}%
\newtheorem{theorem}{Theorem}[section]
\newtheorem{proposition}[theorem]{Proposition}
\newtheorem{corollary}[theorem]{Corollary}
\newtheorem{lemma}[theorem]{Lemma}
\newtheorem{Conjecture}{Conjecture}
{\theoremstyle{definition}
\newtheorem{remark}[theorem]{Remark}
}
{\theoremstyle{definition}%
\newtheorem*{fremark}{Final Remark}
}

\DeclareMathOperator{\spann}{span}
\DeclareMathOperator{\sign}{sign}

\newcommand{\ccc}{\mathbf c}

\def\1{\text{{\textsc{1}}\!\!\!{{\textsc{1}}}}}

\begin{document}

\title[On the $n-$th linear polarization constant of $\mathbb R^n.$]{On the $n-$th linear polarization constant\\ of $\mathbb R^n.$}

\author[D. Pinasco]{Dami\'an Pinasco}
\address{Departamento de Matem\'{a}ticas y Estad\'istica\\
Universidad Torcuato Di Tella\\
Av. Figueroa Alcorta 7350 (C1428BCJ)\\
Buenos Aires, Argentina\\
and CONICET}
\email{dpinasco@utdt.edu}




\renewcommand{\thefootnote}{}




\renewcommand{\thefootnote}{\arabic{footnote}}
\setcounter{footnote}{0}


\maketitle


\begin{center}
{\it To my daughters, Ana and Lara.}
\end{center}

\begin{abstract} We prove that given any set of $n$  unit vectors $\{v_i\}_{i=1}^{n}\subset \mathbb R^n,$ the inequality
\[
\sup\limits_{\Vert x \Vert_{\mathbb R^n} =1} \vert \langle x, v_1 \rangle \cdots \langle x, v_n\rangle\vert \ge n^{-n/2}
\]
holds for $n \le 14.$ Moreover, the equality is attained if and only if  $\{v_i\}_{i=1}^{n}$ is an orthonormal system.
\end{abstract}

\section{Introduction.}

Let $P_1, \ldots, P_n$ be homogeneous polynomials defined on $\mathbb R^m, \mathbb C^m$ or, in general, in any Banach space. Given a norm $\Vert \, \cdot \, \Vert$ defined on the space of polynomials, the problem of finding a constant $M$, depending only on the degrees of $P_1, \ldots, P_n$, such that
\begin{equation*}
\Vert P_1 \Vert \cdots \Vert P_n \Vert \le M \Vert P_1 \cdots P_n\Vert
\end{equation*}
has been extensively studied by many authors. In this note we are concerned with a special case: given any set $\{\phi_i\}_{i=1}^n $ of continuous linear functionals defined on a Hilbert space $H,$ we study the inequality
\begin{equation}\label{problema}
\Vert \phi_1\Vert \cdots \Vert \phi_n\Vert \le M\, \Vert \phi_1 \cdots \phi_n\Vert,
\end{equation}
where $\phi_1 \cdots \phi_n$ is the $n-$homogeneous polynomial defined by the pointwise product
\[
\phi_1 \cdots \phi_n(x)=\phi_1(x) \cdots \phi_n(x),
\]
and $\Vert \cdot \Vert$ is the uniform norm over the unit sphere of $H.$

For a Banach space $E$ with dual space $E'$ and considering the uniform norm on the unit sphere of $E$, C. Ben\'{\i}tez, Y. Sarantopoulos and A. Tonge (see \cite{BST}) defined the $n-$th linear polarization constant of $E$
\begin{align*}
\ccc_n(E)=&\inf\left\{M>0: \Vert \phi_1\Vert \cdots \Vert \phi_n\Vert \le M\, \Vert \phi_1 \cdots \phi_n\Vert, \forall \,  \phi_1, \ldots \phi_n \in E'\right\} \\
=& 1/\inf\left\{\sup\limits_{\Vert x \Vert=1} \vert \phi_1(x) \cdots \phi_n(x)\vert:  \phi_i \in E', \Vert \phi_i\Vert=1\  \forall\ 1 \le i \le n\right\}.
\end{align*}
In~\cite{RT}, R. Ryan and B. Turett, studying the geometry of spaces of polynomials, showed that for each $n$ there is a constant $K_n$ such that $\ccc_n(E) \leq K_n$ for every Banach space $E$. In \cite{BST}, it was proved that the best constant $K_n$ for complex Banach spaces is $n^n$ and S. G. R\'ev\'esz and Y. Sarantopoulos \cite{RS} proved that the best constant $K_n$  for real Banach spaces is also $n^n.$ Note that $\ccc_n(\ell_1^n)=n^n,$ but in general, for different Banach spaces it is possible to find smaller values for $\ccc_n(E).$ In~the last two decades there has been many research articles on this topic, for different techniques and approaches in calculating polarization constants see, for example, \cite{AR},  \cite{F}, \cite{G-VV}, \cite{LeLiRa}, \cite{LMS}, \cite{Ma}, \cite{Ma1}, \cite{MaMu}, \cite{MSS}, \cite{PPT}, \cite{PR} and \cite{RS} and the references therein.

Let $E$ be a Banach space, a {\it plank} or {\it strip} is the set of points between two parallel hyperplanes in $E$. Given a convex body $K \subset E,$ and any norm one linear functional $\phi \in E',$ we can measure the distance between two supporting hyperplanes of $K,$ defined by level sets of $\phi.$ This distance is the width of $K$ in the direction induced by $\phi.$ The {\it minimal width of} $K$ is the minimal width among all directions induced by norm one linear functionals. The following question was posed by A. Tarski (see \cite{Tar2}):

{\it Let $K$ be a convex body covered by $n$ parallel planks, is it true that the sum of the widths of each plank is not less than the minimal width of $K$?}

A positive answer was given by T. Bang in \cite{Ban}, who presented a strengthened version of this question considering the sum of the {\it relative widths} instead of the widths of the planks. This is still an open problem in the general case, but for centrally symmetric bodies a positive answer was given by K. Ball in \cite{Ba1}. It is worth noting that linear polarization constants are related to {\it plank problems} in Banach spaces, in particular the upper bound $\ccc_n(E) \le n^n$ for any real Banach space can be deduced from Theorem 2 in \cite{Ba1}.

In \cite{RS}, using the remarkable theorem of A. Dvoretzky (see \cite{Dv}, \cite{Mi}), it is shown that Hilbert spaces have the smallest $n-$th polarization constant among infinite dimensional Banach spaces. Namely, we have $\ccc_n(\ell_2^n) \le \ccc_n(E)$ for any infinite dimensional Banach space $E.$ So, knowing the exact value of $\ccc_n(\ell_2^n)$ becomes an interesting and intriguing problem. Working in a Hilbert space $H$, by Riesz representation theorem, inequality (\ref{problema}) may be written in the following way
\begin{equation*}\label{hilbert}
\Vert v_1\Vert \cdots \Vert v_n\Vert \le M \sup\limits_{\Vert x \Vert_{H} =1}\, \vert\langle x,v_1\rangle \cdots \langle x,v_n\rangle\vert.
\end{equation*}
Note that we can modify any vector $v_i$ to $-\, v_i$ at our convenience, without altering either sides of the inequality.

Given an orthonormal basis $\{e_i\}_{i=1}^n \subset \mathbb R^n,$ by the Arithmetic-Geometric mean inequality, for any unit vector $x \in \mathbb R^n$  we have
\[
\prod_{i=1}^n \vert \langle x, e_i \rangle \vert = \left(\prod_{i=1}^n \vert \langle x, e_i \rangle \vert^2\right)^{1/2} \le \left(\frac{1}{n}\sum_{i=1}^n \vert \langle x, e_i \rangle \vert^2\right)^{n/2}=n^{-n/2}.
\]
From this bound, it follows that $\ccc_n(\mathbb R^n)\ge \sqrt{n^n}.$
In \cite{BST}, C. Ben\'itez, Y. Sarantopoulos and A. Tonge asked if $\ccc_n(\mathbb R^n)=n^{n/2},$ making the following conjecture.

\begin{Conjecture}\label{pol14} Given $n$ unit vectors $\{v_i\}_{i=1}^n \subset \mathbb R^n$, then
\begin{equation}\label{BST}
\sup\limits_{\Vert x \Vert_{\mathbb R^n} =1} \vert \langle x, v_1 \rangle \cdots \langle x, v_n\rangle\vert \ge n^{-n/2},
\end{equation}
and equality holds if and only if $\{v_i\}_{i=1}^n$ is an orthonormal system.
\end{Conjecture}
For $n\le5,$ the inequality was proved by A. Pappas and S. G. R\'ev\'esz in \cite{PR} (see also \cite{PPT}). However, the question remained unanswered for $n\ge 6.$

A complex analogue of inequality (\ref{BST}) was proved by J. Arias-de-Reyna in \cite{A}. More precisely, the author showed that for any set of unit vectors $\{z_i\}_{i=1}^n \subset \mathbb C^n$ it follows that
\begin{equation}\label{CV}
\sup\limits_{\Vert z \Vert_{\mathbb C^n} =1} \vert \langle z, z_1 \rangle \cdots \langle z, z_n\rangle\vert \ge n^{-n/2}.
\end{equation}
In \cite{Ba}, K. Ball proved a stronger result, which is known as ``{\it the complex plank problem for Hilbert spaces}'' and also implies inequality (\ref{CV}). Finally, in \cite{Pi}, the author showed that a set $\{z_i\}_{i=1}^n$ of unit vectors in a complex Hilbert space $H$ for which the equality is attained must be an orthonormal system.

Although the exact value of the $n-$th linear polarization constant of $\mathbb R^n$ is not known for all $n \in \mathbb N,$ there are several articles in the literature finding upper bounds for its value. In fact, the existence of $c \in\mathbb R$ such that $\ccc_n(\mathbb R^n)\le (cn)^{n/2}$ was studied by many authors: A. E. Litvak, V. D. Milman, and G. Schechtman \cite{LMS} for $c\approx 12,67,$ J. C. Garc\'{\i}a-V\'azquez and R. Villa \cite{G-VV} with $c\approx 3,57,$ S. G. R\'ev\'esz and Y. Sarantopoulos \cite{RS} proved that $\ccc_n(\mathbb R^n)\le 2^{n/2-1}\, n^{n/2},$ P. E. Frenkel \cite{F} improved the previous bound showing that $\ccc_n(\mathbb R^n)\le \left(3^{3/2} e^{-1}\right)^{n/2}\, n^{n/2},$ and G. A. Mu\~{n}oz-Fern\'{a}ndez et al. \cite{MSS} showed that $\ccc_n(\mathbb R^n) \le n2^{n/4}n^{n/2},$ which is asymptotically tighter than the previous bounds.

The aim of this work is to extend the validity range of Conjecture \ref{pol14}. In Section \ref{uno}, given $n \in \mathbb N,$ we study the minima of some constrained problems depending on a parameter $s.$ Namely, we have sets $\Sigma_s \subset \mathbb R^n,$ a function $f:\Sigma_s \to \mathbb R,$ and compute $\min\limits_{a \in \Sigma_s} f(a),$ and then we find the minima of $s \mapsto \min\limits_{a \in \Sigma_s} f(a).$ Finally, in Section \ref{614}, we will apply the results from Section \ref{uno} to prove, for $6\le n \le 14,$ that the $n-$th linear polarization constant of $\mathbb R^n$ is $n^{n/2}.$ Moreover, we show that we will have an equality in (\ref{BST}) if and only if $\{v_i\}_{i=1}^n$ is an orthonormal system.

\section{Some Useful Inequalities and Constrained Problems.}\label{uno}

In this section, for our purposes, we need to present and prove some useful inequalities. We will consider $n\in \mathbb N, n \ge 2, s \in \left[\sqrt{n},n\right]$ and $Q_s=[s^{-1}, 1]^n.$ Given the function $f:Q_s \to \mathbb R,$ defined by $f(a)=a_1 \cdots a_n,$ we are interested in finding the constrained minima
\[
\min_{a\in Q_s} f(a), \text{ subject to } \sum_{i=1}^n a_i=s.
\]
Let us denote by $\Sigma_s$ the set $Q_s \cap \left\{a\in \mathbb R^n : \sum_{i=1}^n a_i=s\right\},$ and define the function $\mu:\left[\sqrt{n},n\right]\to \mathbb R,$ by $\mu(s)=\min\limits_{a\in \Sigma_s} f(a).$

It is easy to see, applying Lagrange multipliers, that
\[
\max_{a\in \Sigma_s} f(a) = f\left(\frac{s}{n},\frac{s}{n}, \ldots, \frac{s}{n}\right)=\left(\frac{s}{n}\right)^n.
\]
So, we might suspect that $\mu(s)$ should be reached on the intersection of the hyperplane and a face of the cube. Moreover, it is reasonable to think that the function $f$ gets smaller as more coordinates of $a$ take the value $s^{-1}.$ Following this idea, let us define
\[
k_0(s)=\min\left\{k\in \mathbb N: n-k < s-k s^{-1}\right\}.
\]
It is clear that $1 \le k_0(s) \le n+1,$ and if $k_0(s) < n,$ this value gives us the first coordinate $k,$ such that any point $a\in Q_s,$
\[a=(\underbrace{s^{-1}, s^{-1}, \ldots, s^{-1}}_{k_0(s) - times}, a_{k_0(s)+1}, \ldots, a_n)
\]
does not belong to $\Sigma_s.$

\begin{remark}\label{k0} Note that from the very definition of $k_0(s),$ we have
\begin{equation}\label{unos}
n-k_0(s) < s-k_0(s) s^{-1},
\end{equation}
which is equivalent to
\[
\frac{s(n-s)}{s-1} < k_0(s).
\]
Also, since
\begin{equation}\label{dos}
s - (k_0(s)-1)s^{-1} \le n+1-k_0(s),
\end{equation}
we obtain
\begin{equation*}
k_0(s)\le \frac{s(n-s)}{s-1}+1.
\end{equation*}
As usual, if $\lfloor x \rfloor=\max\{m\in\mathbb Z: m \le x\}$ denotes the floor function, we can write
\[
k_0(s)=\left\lfloor\frac{s(n-s)}{s-1}\right\rfloor +1.
\]
\end{remark}

\begin{proposition}\label{minima} Let $f:Q_s \to \mathbb R$ be defined by
\[
f(a_1,a_2,\ldots, a_n)=a_1 a_2 \cdots a_n.
\]
Then,
\[
\mu(s)=s^{1-k_0(s)}\left(s-(k_0(s)-1)s^{-1}- n+k_0(s)\right).
\]
\end{proposition}
\begin{proof}
By continuity of $f$ and compactness of $\Sigma_s$ we know that the minimum is attained at some point $a \in \Sigma_s.$ Since $f$ is a symmetric function, we may assume that
\[
s^{-1}\le a_1 \le a_2 \le \ldots \le a_n \le 1.
\]

First, let us show that $a_{k_0(s)-1}=s^{-1}.$ If not, from the definition of $k_0(s),$ we have that $n-(k_0(s)-1) \ge  s-(k_0(s)-1) s^{-1}.$ Then, we obtain
\begin{align*}
s=\sum_{i=1}^n a_i & > (k_0(s)-1)s^{-1} + \left(n+1-k_0(s)\right)a_{k_0(s)} \\
& \ge  s-n-1+k_0(s)+\left(n+1-k_0(s)\right)a_{k_0(s)} \\
& = s+\left(n+1-k_0(s)\right)\left(a_{k_0(s)}-1\right).
\end{align*}
This inequality implies that $a_{k_0(s)}<1.$ Now, taking $$\varepsilon\in \left(0, \min\left(a_{k_0(s)-1}-s^{-1}, 1-a_{k_0(s)}\right)\right),$$ we slightly perturb the values $a_{k_0(s)-1}$ and $a_{k_0(s)}$ by $\varepsilon$:
\begin{align*}
a_1\cdots a_{n} & \le  a_1\cdots \left(a_{k_0(s)-1}-\varepsilon\right)\left(a_{k_0(s)}+\varepsilon\right)\cdots a_{n} \\
& = a_1\cdots \left(a_{k_0(s)-1}a_{k_0(s)}+\varepsilon\left(a_{k_0(s)-1}-a_{k_0(s)}\right)-\varepsilon^2\right)\cdots a_{n} \\
& < a_1\cdots a_{n},
\end{align*}
which is impossible. Then, $a_{k_0(s)-1}=s^{-1}.$

Our next step is to find the minimum of $g: [s^{-1}, 1]^{n+1-k_0(s)}\to\mathbb R,$ defined by
\[
g(a_{k_0(s)}, \ldots, a_n):=f(s^{-1},s^{-1}, \ldots, s^{-1}, a_{k_0(s)}, \ldots, a_n),
\]
subject to the equality constraint
\[
\sum_{i=k_0(s)}^n a_i=s-(k_0(s)-1)s^{-1} \in \left(n-k_0(s)+s^{-1}, n-k_0(s)+1\right).
\]
Write $\widetilde{s}=s-(k_0(s)-1)s^{-1}.$ Let us prove that the minimum is attained at
\[
(a_{k_0(s)}, \ldots, a_n)=\left(\widetilde{s}- n+k_0(s),1,\ldots, 1\right).
\]
Note that if $a_{k_0(s)}> \widetilde{s}- n+k_0(s),$ then it must be $a_{k_0(s)+1}<1,$ and we can proceed as we did in the beginning: i.e. choosing $\varepsilon>0,$ small enough, we can see that
\begin{align*}
a_{k_0(s)}a_{k_0(s)+1}\cdots a_n \le & \left(a_{k_0(s)}-\varepsilon\right)\left(a_{k_0(s)+1}+\varepsilon\right)\cdots a_n \\
 < & \  a_{k_0(s)}a_{k_0(s)+1}\cdots a_n,
\end{align*}
which leads to a contradiction. Therefore, $a_{k_0(s)}=\widetilde{s}- n+k_0(s),$ and $a_{i}=1$ for $k_0(s)+1\le i \le n$. It follows immediately that
\[
\mu(s)=s^{1-k_0(s)}\left(s-(k_0(s)-1)s^{-1}- n+k_0(s)\right). \qedhere
\]
\end{proof}

\begin{corollary}\label{formula}
Let $f:Q_s \to \mathbb R$ be the function defined by $f(a_1,a_2,\ldots,a_n)=a_1a_2\cdots a_n.$ Then,
\[
\mu(s)=\frac{s^{-1}+s - n + \left(\left\lfloor\frac{s(n-s)}{s-1}\right\rfloor +1\right)(1-s^{-1})}{s^{\left\lfloor \frac{s(n-s)}{s-1}\right\rfloor}}.
\]
\end{corollary}

\begin{proof}
The proof is immediate by combining Remark \ref{k0} and Proposition \ref{minima}.
\end{proof}

\begin{remark}\label{mil} Examining inequalities (\ref{unos}) and (\ref{dos}), we deduce that
\[
n-k_0(s)<s - k_0(s)s^{-1}\le n+1-k_0(s)-s^{-1}<n+1-k_0(s).
\]
Then, we obtain
\begin{equation}\label{svida}
 n-k_0(s)=\left\lfloor s - k_0(s)s^{-1}\right\rfloor.
\end{equation}
\end{remark}

\begin{remark}\label{valorsj} The function $\mu$ is continuous for $s\in \left[\sqrt{n},n\right],$ except for those points $\{s_j\}_{j=0}^n$ where $\frac{s_j(n-s_j)}{s_j-1}=j.$ We can compute $s_j$ as the positive root of the equation $x^2-(n-j)x-j=0,$ namely
\[
s_j=\frac{(n-j)+\sqrt{(n-j)^2+4j}}{2}.
\]
Also, since $s \mapsto \frac{s(n-s)}{s-1}$ is a decreasing function, we have
\[
\sqrt{n}=s_n < s_{n-1} < \ldots < s_1< s_0=n.
\]
\end{remark}

\begin{proposition}\label{lsc} The function $\mu: \left[\sqrt{n},n\right] \to \mathbb R$ is lower semi-continuous.
\end{proposition}

\begin{proof} As we noted in Remark \ref{valorsj}, the function $\mu$ is continuous on its domain, except at the points $\{s_j\}_{j=0}^n.$ Then, it remains to show that $\liminf\limits_{s\to s_j} \mu(s)\ge \mu(s_j)$ for $0\le j \le n.$ However, since $s \mapsto \frac{s(n-s)}{s-1}$ decreases, we know that $s \mapsto \left\lfloor \frac{s(n-s)}{s-1}\right\rfloor$ is continuous from the left, so it is not necessary to study the case $j=0$ and we only have to check that $\liminf\limits_{s\to s_j^+} \mu(s)\ge \mu(s_j)$ for $1\le j \le n.$

First, let us compute $\mu(s_j).$ We can write, from Corollary \ref{formula},
\begin{align*}
\mu(s_j)=&\ \frac{s_j^{-1}+s_j - n + \left(\left\lfloor\frac{s_j(n-s_j)}{s_j-1}\right\rfloor +1\right)(1-s_j^{-1})}{s_j^{\left\lfloor \frac{s_j(n-s_j)}{s_j-1}\right\rfloor}} \\
=&\ \frac{s_j^{-1}+s_j - n + \left(j +1\right)(1-s_j^{-1})}{s_j^{j}}\\
= &\ \frac{1+s_j^2 - (n-j-1)s_j-(j +1)}{s_j^{j+1}} \\
=&\ \frac{s_j^2 - (n-j)s_j-j+s_j}{s_j^{j+1}}=\frac{1}{s_j^j},
\end{align*}
where in the last step we have used that $s_j^2-(n-j)s_j-j=0.$

For $s>s_j,$ close enough, $\left\lfloor\frac{s(n-s)}{s-1}\right\rfloor +1=k_{0}(s)=k_0(s_j)-1=j.$ Then, combining with (\ref{svida}), we have
\begin{align*}
\liminf\limits_{s \to s_j^+} \mu(s)=&\liminf\limits_{s \to s_j^+}\frac{s^{-1}+s - js^{-1}-\left\lfloor s - js^{-1}\right\rfloor }{s^{j-1}} \\
\ge & \liminf\limits_{s \to s_j^+}\frac{s^{-1}}{s^{j-1}}=\frac{1}{s_j^j}=\mu(s_j). \qedhere
\end{align*}
\end{proof}

The following lemma is crucial to determine the minimum of the function $\mu.$

\begin{lemma}\label{auxiliar} Given $j, n\in \mathbb N,$ such that $n \ge 2$ and $1\le j \le n,$ then the function $M_j: [s_{j}, s_{j-1}] \to \mathbb R,$ defined by
\[
M_j(x)=x^{2-j}+(j-n)x^{1-j}+(1-j)x^{-j}
\]
is a quasi-concave function.
\end{lemma}
\begin{proof} First, note that for $j=1$ and $j=2,$ the function $M_j$ is concave and the lemma follows. For the remaining cases, when $3\le j\le n,$ we will prove that $M_j$ satisfies one of the following conditions:
\begin{itemize}
\item $M_j$ is an increasing function on $[s_{j}, s_{j-1}].$
\item There exists $t_j \in (s_{j}, s_{j-1})$ such that $M_j$ is an increasing function on $[s_{j}, t_j]$ and it is a decreasing function on $[t_j, s_{j-1}].$
\end{itemize}

Let us compute
\begin{align*}
M_j^{\prime}(x)&=(2-j)x^{1-j}+(j-n)(1-j)x^{-j}-j(1-j)x^{-j-1} \\
&=\frac{(2-j)x^{2}+(j-n)(1-j)x-j(1-j)}{x^{j+1}}\\
&=\frac{x^2+(1-j)\left(x^{2}+(j-n)x-j\right)}{x^{j+1}}.
\end{align*}
Then, to determine the behaviour of $M_j$ it will be enough to analyze the sign of the concave quadratic function
\[
x \mapsto x^2+(1-j)\left(x^{2}+(j-n)x-j\right).
\]
By definition, $s_{j}^{2}+(j-n)s_j-j=0.$ We have $M_j'(s_j)>0,$ then $M_j$ has at most one critical point in the interval $\left[s_{j}, s_{j-1}\right],$ and the assertion is proved.
\end{proof}

Finally, we need the following proposition in order to prove the main theorem of this section.

\begin{proposition}\label{final} Given $n \in \mathbb N,\ 2 \le n \le 14,$ and $0 \le j \le n,$ then $s_j^j \le \sqrt{n^n}.$
\end{proposition}

\begin{proof} From Remark \ref{valorsj} we know that
\[
\sqrt{n}=s_n < s_{n-1} < \ldots < s_1< s_0=n,
\]
then
\begin{equation*}\label{lamitad}
\left(\sqrt{n}\right)^{n/2}=\left(s_n\right)^{n/2} <\left(s_{n-1}\right)^{n/2} < \ldots < \left(s_1\right)^{n/2}<\left(s_0\right)^{n/2}=n^{n/2}.
\end{equation*}

\medskip

\noindent The last inequalities show that $s_j^j \le \sqrt{n^n}$ for $j \le \lfloor n/2 \rfloor.$ Then, we can restrict ourself to $3 \le n \le 14,$ and $\lfloor n/2 \rfloor +1  \le j \le n-1.$ Note that, for such values of $j,$
\[
\sqrt{n}<n^{\frac{n}{2(n-1)}} \le n^{\frac{n}{2j}} \le n^{\frac{n}{2(\lfloor n/2 \rfloor +1)}} < n.
\]
We will prove that
\begin{equation}\label{crece}
\left(n^{n/2j}\right)^2-(n-j)n^{n/2j}-j \ge 0,
\end{equation}
which is equivalent to $s_j \le n^{n/2j}.$
Let us write inequality (\ref{crece}) as follows
\begin{equation}\label{crece2}
\left(n^{n/2j}\right)^2-nn^{n/2j} \ge (1-n^{n/2j})j.
\end{equation}
If we call $x=n^{n/2j},$ then we have $j=\frac{\ln\left(\sqrt{n^n}\right)}{\ln(x)},$ and inequality (\ref{crece2}) becomes
\begin{equation}\label{crece3}
\frac{1-x}{x^2-nx} - \frac{\ln(x)}{\ln\left(\sqrt{n^n}\right)} \ge 0 \quad \mbox{ for } x \in J_n=\left[n^{\frac{n}{2(n-1)}}, n^{\frac{n}{2(\lfloor n/2 \rfloor +1)}}\right].
\end{equation}
Note that $J_n \subset \left(\sqrt{n},n\right)$ and, in order to prove inequality (\ref{crece3}), we may define the function $\phi:(0,n) \to \mathbb R$ by
\[
\phi(x)=\frac{1-x}{x^2-nx} - \frac{\ln(x)}{\ln\left(\sqrt{n^n}\right)},
\]
and show that $\phi\left(n^{\frac{n}{2(n-1)}}\right)\ge 0$ and that $\phi(x)$ is an increasing function on $J_n.$ For this purpose we can study
\[
\phi^{\prime}(x)=\frac{x^2-2x+n}{x^2(x-n)^2}-\frac{1}{x \ln\left(\sqrt{n^n}\right)}.
\]
To prove that $\phi^{\prime}(x) > 0,$ we may write
\[
\phi^{\prime}(x)=\frac{1}{x \ln\left(\sqrt{n^n}\right)}\left(\ln\left(\sqrt{n^n}\right)\ \frac{x^2-2x+n}{x(x-n)^2}-1\right),
\]
and prove that
\[
\ln\left(\sqrt{n^n}\right)\ \frac{x^2-2x+n}{x(x-n)^2}> 1 \quad \mbox{ for all } x \in J_n.
\]
Let us analyze the function $\varphi(x)=\dfrac{x^2-2x+n}{x(x-n)^2}.$ Its derivative is just
\[
\varphi^{\prime}(x)=\dfrac{x^3+\left(n-4\right)x^2+3nx-n^2}{x^2\left(n-x\right)^3}.
\]
The sign of $\varphi^{\prime}(x)$ over the interval $J_n$ depends on the sign of the function
\[
x \mapsto x^3+\left(n-4\right)x^2+3nx-n^2.
\]
But this is a monotone increasing function (for $2 \le n \le 14$) because its derivative, $3x^2+2(n-4)x+3n,$ has no real roots and is positive on $\mathbb R.$  In fact, the discriminant $\Delta=4(n^2-17n+16)$ could be factored as $\Delta=4(n-1)(n-16),$ which is negative for $2 \le n \le 15.$ Then, for $x \in \left[\sqrt{n},n\right)$ we obtain
\begin{align*}
\sign\left(\varphi^{\prime}(x)\right)=&\ \sign\left(\sqrt{n}^3+\left(n-4\right)n+3n\sqrt{n}-n^2\right)\\
=& \ \sign\left(4\left(\sqrt{n^3}-n\right)\right)=1.
\end{align*}
Since $\varphi$ is an increasing function on $J_n,$ we may show that $\ln\left(\sqrt{n^n}\right) \varphi\left(n^{\frac{n}{2(n-1)}}\right)>1,$ i.e.
\[
\ln\left(\sqrt{n^n}\right)\ \frac{\left(n^{\frac{n}{2(n-1)}}\right)^2-2\left(n^{\frac{n}{2(n-1)}}\right)+n}{\left(n^{\frac{n}{2(n-1)}}\right)\left(\left(n^{\frac{n}{2(n-1)}}\right)-n\right)^2}> 1,
\]
to conclude that $\phi'(x)>0$ on $J_n.$ Once this is done, it remains to check that $\phi\left(n^{\frac{n}{2(n-1)}}\right)>0$ to ensure that the inequality (\ref{crece3}) is satisfied. But, as we said, this is equivalent to show that ${s_{n-1}^{n-1}\le \sqrt{n^n}.}$

The following table contains these values for $3\le n \le 16.$

\begin{table}[!ht]
\begin{center}
\scalebox{0.95}{
\begin{tabular}{|c|c|c|c|}
\hline
& & & \\
$n$ & $\ln(\sqrt{n^n})\ \frac{\left(n^{\frac{n}{2(n-1)}}\right)^2-2\left(n^{\frac{n}{2(n-1)}}\right)+n}{\left(n^{\frac{n}{2(n-1)}}\right)\left(\left(n^{\frac{n}{2(n-1)}}\right)-n\right)^2}$ & $S_{n-1}^{n-1}$ & $\sqrt{n^n}$\\
& & & \\ \hline
$3$ & $\approx 5,065$ & 4 & $\approx 5,196$ \\ \hline

$4$ & $\approx 2,666$ & $\approx 12,211$ & $16$  \\ \hline

$5$ & $\approx 2,008$ & $\approx 43,053$  & $\approx 55,901$\\ \hline

$6$ & $\approx 1,698$ & $\approx 169,442$  & $216$ \\ \hline

$7$ & $\approx 1,514$ & $729 $ & $\approx 907,492$\\ \hline

$8$ & $\approx 1,389$ & $\approx 3380,607$	 & $4096$ \\ \hline

$9$ & $\approx 1,298$ & $\approx  16725,933$  & $19683$ \\ \hline

$10$ & $\approx 1,227$ & $\approx 87610,098$ & $100000$ \\ \hline

$11$ & $\approx 1,170$ & $\approx 482892,455$ & $\approx 534145,739$ \\ \hline

$12$ & $\approx 1,123$ & $\approx 2787117,027$ &  $2985984$ \\ \hline

$13$ & $\approx 1,084$ & $16777216$ & $\approx 17403307,350$ \\ \hline

$14$ & $\approx 1,049$ & $\approx 104973424,100$ & $105413504$ \\ \hline

$15$ & $\approx 1,019$ & $\approx 680750436,468$ & $\approx 661735513,918$ \\ \hline

$16$ & $\approx 0,992$ & $4564290812,351$ & $4294967296$ \\ \hline
\end{tabular}
}
\end{center}
\end{table}
The second columm shows that, for $3\le n \le 15, \ln\left(\sqrt{n^n}\right)\, \varphi\left(n^{\frac{n}{2(n-1)}}\right)$ is greater than 1. Comparing the third with the fourth column, we see that $s_{n-1}^{n-1}$ is less than or equal to $\sqrt{n^n}$ for $3\le n \le 14.$ As both inequalities are fulfilled for $3\le n \le 14,$ the assertion is proved.
\end{proof}

\begin{remark} Note that for $n=15$ and $n=16$ we have ${s_{n-1}^{n-1} >\sqrt{n^{n}}}.$ Moreover, for $n=16$ we can have $\phi'(x)<0$ at some points.
\end{remark}

\begin{remark}\label{ocond} Actually we have proved that $s_j^j < \sqrt{n^n},$ unless $j=n.$
\end{remark}

\begin{theorem}\label{minimo} Given $n \in \mathbb N,\ 2 \le n \le 14,$ let $f:Q_s \to \mathbb R$ be the function defined by $f(a_1,a_2,\ldots, a_n)=a_1 a_2 \cdots a_n.$ If we consider $\mu: \left[\sqrt{n},n\right] \to \mathbb R$, where ${\mu(s)=\min\limits_{a\in \Sigma_s} f(a),}$ then
\[
\mu(s) \ge \frac{1}{\sqrt{n^n}}.
\]
Moreover, the minimum is attained only at $s=\sqrt{n}.$
\end{theorem}

\begin{proof} Let us begin by restricting $\mu$ to the open interval $I_j=(s_j, s_{j-1}),$ for ${1 \le j \le n.}$  From Corollary \ref{formula} we can write
\[
\mu(s)=\frac{s^{-1}+s - n + \left(\left\lfloor\frac{s(n-s)}{s-1}\right\rfloor +1\right)(1-s^{-1})}{s^{\left\lfloor \frac{s(n-s)}{s-1}\right\rfloor}}.
\]
Since $\left\lfloor\frac{s(n-s)}{s-1}\right\rfloor +1=j$ for any $s \in I_j,$ we obtain
\[
(\mu_{|I_j})(s)=\frac{s^{-1}+s - n + j(1-s^{-1})}{s^{j-1}}=s^{2-j}+(j-n)s^{1-j}+(1-j)s^{-j}.
\]

Recall that $\mu$ is a lower semi-continuous function, and it must attain its minimum on any compact set. Therefore, there exists some point ${\bf\sc{s_j}}\in [s_j, s_{j-1}],$ such that $\mu(s)\ge \mu({\bf\sc{s_j}})$ for all $s \in [s_j, s_{j-1}].$

For any $s \in (s_j, s_{j-1}),$ the evaluation of $\mu(s)$ coincides with the evaluation of the  function $M_j: \left[s_{j}, s_{j-1}\right] \to \mathbb R,$ considered in Lemma \ref{auxiliar}. Then it follows that ${\bf\sc{s_j}}$ does not belong to the open interval $(s_j, s_{j-1}).$

Since
\[
\min_{s\in \left[\sqrt{n},n\right]} \mu(s)=\min_{1\le j \le n}\ \min_{s\in [s_j, s_{j-1}]} \mu(s),
\]
the minimum of $\mu(s): \left[\sqrt{n},n\right] \to \mathbb R$ must be attained at ${\bf\sc{s}} \in \{s_j\}_{j=0}^n,$ and it suffices to prove
\[
\mu(s_j) \ge \dfrac{1}{\sqrt{n^n}}
\]
for $0 \le j \le n.$ Then, the proof follows from Proposition \ref{final} and Remark \ref{ocond}.
\end{proof}

\section{The $n-$th linear polarization constant of $\mathbb R^n.$}\label{614}

Let us begin by recalling that in order to show the equality $\ccc_n(\mathbb R^n)=\sqrt{n^n},$ it is enough to prove that for any set of unit vectors $\{v_i\}_{i=1}^n \subset \mathbb R^n,$ there exists a norm one vector $x \in \mathbb R^n$ such that
\begin{equation}\label{PR}
\vert \langle x, v_1 \rangle \cdots \langle x, v_n\rangle\vert \ge n^{-n/2}.
\end{equation}
In \cite{PR}, the authors ensure the existence of a norm one vector $x \in \mathbb R^n$ satisfying inequality (\ref{PR}) for $n=2, 3, 4$ and $5.$ The proof is based on an appropriate choice of signs $\{\varepsilon_i\}_{i=1}^n$ such that maximizes the euclidean norm of $\sum_{i=1}^n \varepsilon_i v_i.$ Then, the desired vector is
\[
x=\frac{\sum_{i=1}^n \varepsilon_i v_i}{\left\Vert \sum_{i=1}^n \varepsilon_i v_i\right\Vert}.
\]

\bigskip

Note that for any choice of signs we have
\[
\left\Vert \sum_{i=1}^n \varepsilon_i v_i \right\Vert_2^2=\sum_{i=1}^n\left\Vert \varepsilon_i v_i \right\Vert_2^2+ 2 \sum_{1\le i \neq j \le n} \varepsilon_i \varepsilon_j \langle v_i, v_j \rangle.
\]
If we consider the random vector of signs $(\varepsilon)_j=(\varepsilon_{j_1}, \ldots, \varepsilon_{j_n}),$ all with equal probability of being chosen, the mean of the squared norm is
\[
\frac{1}{2^n}\sum_{j=1}^{2^n} \ \left\Vert \sum_{i=1}^n \varepsilon_{j_i} v_i \right\Vert_2^2=n.
\]
For our purposes we may assume that the choice of signs maximizing $\left\Vert \sum_{i=1}^n \varepsilon_i v_i \right\Vert_2^2$ is just $(\varepsilon)_j=(1, \ldots, 1).$ In the sequel we will consider sets of unit vectors $\{v_i\}_{i=1}^n$ such that the {\it longest sum} of them is $v=\sum_{i=1}^n v_i.$ Of course, it satisfies
\[
\sqrt{n} \le \Vert v \Vert \le n.
\]
For this vector $v,$ we have $\langle v,v\rangle \ge \langle v-2v_i,v-2v_i\rangle$ for all $1 \le i \le n.$ Then,
\[
\Vert v \Vert^2 \ge \Vert v \Vert^2 - 4 \langle v_i, v \rangle + 4.
\]
It follows that $\langle v_i, v \rangle \ge 1$ for $1 \le i \le n$ (see \cite{PR} for further details).

\bigskip

Although the following result is known for $2 \le n \le 5$ (see \cite{PR}), we include these cases in the statement of our main theorem.
\begin{theorem} Given $2 \le n \le 14,$ then $\ccc_n(\mathbb R^n)=\sqrt{n^n}.$
\end{theorem}

\begin{proof} Take any set of unit vectors $\{v_i\}_{i=1}^n \subset \mathbb R^n,$ such that $v=\sum_{i=1}^n v_i$ is the {\it longest sum} of them. Let us show that
\[
\prod_{i=1}^n \left\langle v_i,\frac{v}{\Vert v \Vert} \right\rangle \ge \frac{1}{\sqrt{n^n}}.
\]
Write $\langle v_i, v \rangle = a_i \Vert v \Vert \ge 1,$ for some $a_i \in \mathbb R.$ Then $a_i \ge\Vert v \Vert^{-1}$ and, from Cauchy-Schwarz inequality, $a_i \le \Vert v_i \Vert =1.$ Also,
\[
\sum_{i=1}^n a_i=\sum_{i=1}^n \left\langle v_i, \frac{v}{\Vert v \Vert}\right\rangle=\left\langle v, \frac{v}{\Vert v \Vert}\right\rangle= \Vert v \Vert \in \left[\sqrt{n},n\right].
\]
Now, applying Theorem \ref{minimo} for $s=\Vert v \Vert,$ we obtain
\[
\prod_{i=1}^n \left\langle v_i,\frac{v}{\Vert v \Vert} \right\rangle=f(a_1, \ldots, a_n) \ge \mu(\Vert v\Vert) \ge \frac{1}{\sqrt{n^n}}. \qedhere
\]
\end{proof}

\begin{lemma} Let $\{v_i\}_{i=1}^n \subset \mathbb R^n$ be unit vectors such that for any choice of signs $\varepsilon_i,$ we have $\left\Vert \sum_{i=1}^n \varepsilon_i v_i \right\Vert^2=n,$ then $\{v_i\}_{i=1}^n$ is an orthonormal system.
\end{lemma}
\begin{proof} Let us call $\mathfrak{I}=\{1,2, \ldots, n\}.$ For $j\in \mathfrak{I},$ let $\mathfrak{I}_j$ be the set $\mathfrak{I}-\{j\}.$
Since
\[
\left\Vert \sum_{i\in \mathfrak{I}_j} \varepsilon_i v_i + v_{j}\right\Vert^2=\left\Vert \sum_{i\in \mathfrak{I}_j} \varepsilon_i v_i - v_{j}\right\Vert^2 \\
\]
it follows that
\[
\left\langle\sum_{i\in \mathfrak{I}_j} \varepsilon_i v_i, v_{j}\right\rangle=0
\]
for any choice of signs $\{\varepsilon_i\}_{i\in \mathfrak{I}_j} \subset \{-1, 1\}^{n-1}.$ Finally, since
\[
\spann \left\{\sum_{i\in \mathfrak{I}_j} \varepsilon_i v_i: \varepsilon_i \in \{-1, 1\} \right\}=\spann \left\{v_i:  i\in \mathfrak{I}_j\right\},
\]
we deduce that $\langle v_{i}, v_j\rangle=0$ for $i \in \mathfrak{I}_j.$ Now, since we can freely choose $j \in \mathfrak{I},$ the lemma is proved.
\end{proof}

\begin{theorem}  For $2 \le n \le 14,$ if $\{v_i\}_{i=1}^n \subset \mathbb R$ are unit vectors such that
\[
\sup\limits_{\Vert x \Vert =1} \vert \langle x, v_1 \rangle \cdots \langle x, v_n\rangle\vert = n^{-n/2},
\]
then $\{v_i\}_{i=1}^n$ is an orthonormal system.
\end{theorem}

\begin{proof} Recall, from Remark \ref{ocond}, that the value $\sqrt{n^n}$ is attained only if $s=\sqrt{n}.$ Now, the {\it longest sum} $v=v_1+\ldots +v_n$ must have norm $\sqrt{n}.$ But then, any vector $v_{(\varepsilon)_j}=\sum_{i=1}^n \varepsilon_{j_i} v_i$ has norm $\sqrt{n}.$ Hence, the assertion follows from the previous lemma.
\end{proof}

\begin{fremark} From \cite{PR} we knew that the {\it longest sum} of a set of vectors was a good candidate to check if inequality (\ref{BST}) holds. In this work we extended from $n=5$ to $n=14$ the validity range of Conjecture \ref{pol14}, by applying the results from Section \ref{uno}. For $n=34,$ M. Matolcsi and G. A. Mu\~{n}oz (see \cite{MaMu}) gave an example where the {\it longest sum} $v$ of some set of unit vectors $\{v_i\}_{i=1}^{34}$ does not satisfy the inequality $\prod \vert \langle v_i,v\rangle\vert ~\ge~34^{-17}.$ However, the inequality holds in some alternative vector. From their example it is possible to construct many others for any $n>34.$

Given $s \in \left[\sqrt{n},n\right],$ if we denote by $\mathcal F(s)$ the set of all $n-$tuples of unit vectors $\{v_i\}_{i=1}^n,$ such that its {\it longest sum} $v$ has $\Vert v \Vert=s,$ then the map
\[
\Lambda: \mathcal F(s) \longrightarrow \Sigma_{s}
\]
defined by
\[
\Lambda(v_1, v_2, \ldots, v_n)=\frac{1}{s}\left(\langle v_1,v\rangle, \ldots, \langle v_n,v\rangle \right)
\]
is not necessarily surjective. Then it is possible that the {\it longest sum} $v$ still works as a good tester for Inequality (\ref{BST}) for some other values of $14<n<34.$
\end{fremark}

\section*{Acknowledgements} The author wishes to express his gratitude to Vicky Venuti for many stimulating conversations during the preparation of this document. Also, the author would like to thank the referees for carefully reading the manuscript, and for their useful comments and remarks which improved this article.


\begin{thebibliography}{10}
\bibitem{AR} V. A. Anagnostopoulos and S. G. R\'ev\'esz, {\it Polarization constants for products of linear functionals over ${\mathbb R}\sp 2$ and ${\mathbb C}\sp 2$ and Chebyshev constants of the unit sphere.} {Publ. Math. Debrecen} {\bf 68} (2006), no. 1-2, 63--75.
\bibitem{A} J. Arias-de-Reyna, {\it Gaussian variables, polynomials and permanents.} {Linear Algebra Appl.} {\bf 285} (1998), no. 1-3, 107--114.
\bibitem{Ba} K. M. Ball, {\it The complex plank problem.} {Bull. London Math. Soc.} {\bf 33} (2001), no. 4, 433--442.
\bibitem{Ba1} K. M. Ball, {\it The plank problem for symmetric bodies}. Invent. Math. {\bf 104} (1991), no. 3, 535--543.
\bibitem{Ban} T. Bang, {\it A solution of the ``plank problem''}. Proc. Amer. Math. Soc. {\bf 2} (1951), 990--993.
\bibitem{BST} C. Ben\'{\i}tez, Y. Sarantopoulos and A. Tonge, {\it Lower bounds for norms of products of polynomials.} {Math. Proc. Cambridge Philos. Soc.} {\bf 124} (1998), no. 3, 395--408.
\bibitem{Dv} A. Dvoretzky, {\it Some results on convex bodies and Banach spaces.} 1961 Proc. Internat. Sympos. Linear Spaces (Jerusalem, 1960), 123--160
\bibitem{F} P. E. Frenkel, {\it Pfaffians, Hafnians and products of real linear functionals.} {Math. Res. Lett.} {\bf 15} (2008), no. 2, 351--358.
\bibitem{G-VV} J. C. Garc\'{\i}a-V\'azquez and R. Villa, {\it Lower bounds for multilinear forms defined on Hilbert spaces.} Mathematika {\bf 46} (1999), no. 2, 315--322.
\bibitem{LeLiRa} Y. J. Leung, W. V. Li and Rakesh, {\it The {$d$}th linear polarization constant of {$\mathbb R^d.$}} J. Funct. Anal. {\bf 255} (2008), no. 10, 2861--2871.
\bibitem{LMS} A. E. Litvak, V. D. Milman and  G. Schechtman, {\it Averages of norms and quasi-norms}. Math. Ann. {\bf 312} (1998), no. 1, 95--124.
\bibitem{Ma} M. Matolcsi, {\it A geometric estimate on the norm of product of functionals.} {Linear Algebra Appl.} {\bf 405} (2005), 304--310.
\bibitem{Ma1} M. Matolcsi, {\it The linear polarization constant of $\mathbb R\sp n$.} {Acta Math. Hungar.} {\bf 108} (2005), no. 1-2, 129--136.
\bibitem{MaMu} M. Matolcsi and G. A. Mu\~{n}oz, {\it On the real linear polarization constant problem.} {Math. Inequal. Appl.} {\bf 9} (2006), no. 3, 485--494.
\bibitem{Mi} V. D. Milman, {\it A new proof of A. Dvoretzky's theorem on cross-sections of convex bodies.} Funkcional. Anal. I Prilozhen (in Russian) {\bf 5} (1971), no. 4, 28--37.
\bibitem{MSS} G. A. Mu\~{n}oz-Fern\'{a}ndez, Y. Sarantopoulos and J. B. Seoane-Sep\'{u}lveda, J. B., {\it The real plank problem and some applications.} Proc. Amer. Math. Soc. {\bf 138} (2010), no. 7, 2521--2535.
\bibitem{PPT} A. Pappas, P. Papadopoulos and E. Theofili, {\it The optimal lower bound for a polynomial norm which is a product of linear and continuous forms in a Hilbert space}. Non-linear Functional Analysis and Applications {\bf 20} (2015), no. 1, 79--95.
\bibitem{PR} A. Pappas and S. G. R\'ev\'esz, {\it Linear polarization constants of Hilbert spaces.} {J. Math. Anal. Appl.} {\bf 300} (2004), no. 1, 129--146.
\bibitem{Pi} D. Pinasco, {\it Lower bounds for norms of products of polynomials via Bombieri inequality.} {Trans. Amer. Math. Soc.} {\bf 364} (2012), no. 8, 3993--4010.
\bibitem{RS} S. G. R\'ev\'esz and Y. Sarantopoulos, {\it Plank problems, polarization and Chebyshev constants.} Satellite Conference on Infinite Dimensional Function Theory. {J. Korean Math. Soc.} {\bf 41} (2004), no. 1, 157--174.
\bibitem{RT} R. Ryan and B. Turett, {\it Geometry of spaces of polynomials.} {J. Math. Anal. Appl.} {\bf 221} (1998), no. 2, 2698--711.
\bibitem{Tar2} A. Tarski, {\it Uwagi o stopniu r\'ownowa\.{o}no\'sci wielok\c{a}t\'ow}. (English: Remarks on the degree of equivalence of polygons) Parametr 2 (1932), 310--314.
\end{thebibliography}
\end{document}